\documentclass[12pt]{amsart} % Aumenta o tamanho da fonte para 12pt

\usepackage{amsmath,amsfonts,amssymb,amsthm}
\usepackage{hyperref}
\usepackage{geometry} % Permite ajustar as margens
\usepackage{tikz}
\usepackage{amsmath}
\usepackage{frcursive}
\usepackage{calligra}
\usepackage{mathrsfs}
\usepackage{eurosym}
\usepackage{amssymb}
\usepackage{graphicx}
\usepackage{xcolor}

\newtheorem{theorem}{Theorem}[section] % Define o ambiente de Teorema
\newtheorem{corollary}[theorem]{Corollary} % Define o ambiente de Corolário
\newtheorem{question}[theorem]{Question} % Define o ambiente de Pergunta
\newtheorem{lemma}[theorem]{Lemma} % Define o ambiente de Lema
\newtheorem{proposition}[theorem]{Proposition} % Define o ambiente de Proposição
\theoremstyle{definition} % Altera o estilo para definição
\newtheorem{definition}[theorem]{Definition} % Define o ambiente de Definição
\newtheorem{example}[theorem]{Example} % Define o ambiente de Exemplo
\theoremstyle{remark} % Altera o estilo para observação
\newtheorem{remark}[theorem]{Remark} % Define o ambiente de Observação

\newcommand{\Ker}{\text{\normalfont{Ker}}}
\newcommand{\ad}{\text{\normalfont{ad}}}

\newcommand{\Ko}{K^{\text{\normalfont{o}}}}
\newcommand{\Ke}{K^{\text{\normalfont{e}}}}
\newcommand{\Sym}{\text{\normalfont{Sym}}}
\newcommand{\Dera}{\text{\normalfont{Dera}}}
\newcommand{\Lie}{\text{\normalfont{Lie}}}
\newcommand{\Ends}{\text{\normalfont{End}}^{S}}
\newcommand{\Enda}{\text{\normalfont{End}}^{A}}
\newcommand{\End}{\text{\normalfont{End}}}
\newcommand{\pr}{\text{\normalfont{pr}}}
\newcommand{\GL}{\text{\normalfont{GL}}}
\newcommand{\gl}{\mathfrak{gl}}

\newcommand{\Auto}{\text{\normalfont{Auto}}}
\newcommand{\Iso}{\text{\normalfont{Iso}}}
\newcommand{\Ad}{\text{\normalfont{Ad}}}

\newcommand{\deri}{\text{\normalfont{d}} }
\newcommand{\Id}{\text{\normalfont{Id}} }

\newcommand{\mg}{\mathfrak g }
\newcommand{\so}{{\mathfrak s }{\mathfrak o }}

\newcommand{\mh}{\mathfrak h }

\newcommand{\Kil}{\mathscr{K} }

\renewcommand{\d}{{\rm d}}

\renewcommand{\Im}{{\rm Im\,}}
\newcommand{\R}{\mathbb R}

\title[Symmetric Killing tensors on almost abelian Lie groups]{SYMMETRIC KILLING TENSORS ON ALMOST ABELIAN LIE GROUPS}
\author{Renan Berto-Cuevas}
\address{R.~Berto-Cuevas: Instituto de Matemática, Estatística e Computação Científica, Universidade Estadual de Campinas,  Rua Sergio Buarque de Holanda, 651, Cidade Universitaria Zeferino Vaz, 13083-859, Campinas, São Paulo, Brazil.}
\email{renanbertocuevas@ime.unicamp.br}

\author{Viviana del Barco}
\address{V.~del Barco: Instituto de Matemática, Estatística e Computação Científica, Universidade Estadual de Campinas,  Rua Sergio Buarque de Holanda, 651, Cidade Universitaria Zeferino Vaz, 13083-859, Campinas, São Paulo, Brazil.}
\email{delbarc@ime.unicamp.br}

\author{Andrei Moroianu}
\address{A.~Moroianu: Université Paris-Saclay, CNRS,  Laboratoire de mathématiques d'Orsay, 91405, Orsay, France, and Institute of Mathematics “Simion Stoilow” of the Romanian Academy, 21 Calea Grivitei, 010702 Bucharest, Romania}
\email{andrei.moroianu@math.cnrs.fr}

\subjclass[2020]{22E25, 53C30, 53D25} 
\keywords{Symmetric Killing tensor, Almost abelian Lie group} 

\begin{document}
	
	\begin{abstract}In this work we provide a complete characterization of left-invariant symmetric Killing 
tensors on almost abelian Lie groups endowed with a left-invariant Riemannian metric. We show in particular that all such tensors are decomposable, in the sense that they can be expressed as a polynomial in the Killing vector fields and the Riemannian metric.
	\end{abstract}
	
	\maketitle

\section{Introduction}
On a Riemannian manifold $(M,g)$, a symmetric $p$-tensor $K \in \Gamma(\Sym^{p}(TM))$ is said to be Killing if it is constant along geodesics, in the sense that for every geodesic $\gamma:I\to M$, the function $(\gamma')^{p} \lrcorner (K\circ \gamma)$ is constant on $I$. Equivalently, $K$ is Killing if and  only if its covariant derivative with respect to the Levi-Civita connection $\nabla$ of $g$ satisfies $$(\nabla_X K)(X,X,\ldots, ,X)=0,$$ for every $X\in \mathfrak X(M)$. 
 For $p=1$, one recovers the definition of Killing vector fields on $M$.

 Symmetric Killing tensors were first investigated for their applications in physics, where their relevance arose in the study of integrable systems: indeed, symmetric Killing tensors generate first integrals of the equations of motion. More recently, they have become a topic of interest in their own right, as in \cite{HMS16,semmelmann2002conformal}, and their conformal generalizations have been explored in connection with inverse problems, see \cite{GPSU16,PSU15}.

The simplest symmetric Killing tensors on a given Riemannian manifold are parallel tensors (in particular, the metric)  and Killing vector fields. A way of producing symmetric Killing tensors is as polynomials (i.e. sum of symmetric products) in these elementary symmetric Killing tensors. Then, a natural question arises:  on a given Riemannian manifold, can every symmetric Killing tensor be produced by this method? This is referred to as the decomposability problem. Some cases for which the answer to this question is positive were provided by \cite{takekilli,Thomps}, where it is shown that on spaces of constant sectional curvature, every symmetric Killing tensor can be written as a polynomial in Killing vector fields. On the contrary in \cite{Kol83, mat11} it is shown that some specific metrics on surfaces carry {\em indecomposable} symmetric Killing 2-tensors. However, the question is still open in higher degrees: there is no example of Riemannian surface carrying indecomposable symmetric Killing $p$-tensors for $p\ge 3$.

The papers \cite{HMS16,semmelmann2002conformal} contain a suitable formalism to deal with Killing tensors, and their conformal generalizations. This formalism was applied by the last two named authors to study left-invariant symmetric Killing 2-tensors on 2-step nilpotent Lie groups. More precisely, in \cite{dBM} it is shown that 
on every 2-step nilpotent Lie group of dimension $\leq$ 7,  every left-invariant symmetric Killing 2-tensor  is decomposable. Moreover, the upper bound on the dimension is sharp since there exist 2-step nilpotent Lie groups carrying families of indecomposable symmetric Killing 2-tensors in every dimension $\geq 8$.  First integrals of the geodesic flow on Lie groups endowed with left-invariant metrics, and on their compact quotients by lattices, have been studied by several authors \cite{Butler03,EB, KOR, Ov20,Ov22}.

The goal of this paper is to study left-invariant symmetric Killing tensors on a different class of solvable Lie groups, namely the almost abelian ones. Recall that a Lie group is called almost abelian if its Lie algebra contains a codimension one abelian ideal. The structure of their Lie algebras, which are called almost abelian as well, is encoded in an endomorphism $D\in \gl(n)$ (see Section \ref{sec:almostab} for details). Since the symmetric Killing tensors under study are left-invariant, we consider them as linear tensors on the Lie algebra.

In Theorem \ref{biggestthr}, we provide a characterization of left-invariant symmetric Killing $p$-tensors, for $p\geq 0$, that depends on the endomorphism $D$. We rely on this characterization to prove the main result of our paper, which is Theorem \ref{teo:main}: any left-invariant symmetric  Killing tensor on an almost abelian Lie group is decomposable, that is, it can be written as a polynomial in the metric and Killing vector fields. 

The paper concludes with a discussion of almost abelian Lie groups admitting a left-invariant metric of constant (non-positive) sectional curvature. The curvature constraint implies that every symmetric Killing tensor is a polynomial in Killing vector fields by \cite{takekilli,Thomps}. We are mainly interested in determining if the Killing vector fields that are linked to the algebraic structure, like left- and right-invariant Killing vector fields, are enough to generate all left-invariant Killing tensors. We show that this is the case for zero curvature in Proposition \ref{pro:flat}, but fails to be true otherwise (see Theorem \ref{teo:seccte}). 

\medskip

\noindent {\bf Acknowledgments:} R.~B-C. was partly supported by FAPESP grants 2022/11934-3 and 2023/03360-0 and by the Coordenação de Aperfeiçoamento de Pessoal de Nível Superior - Brasil (CAPES) - Finance Code 001. V.~dB. acknowledges financial support by FAPESP grants 2024/19272-5 and 2023/15089-9. A.~M. was partly supported by the PNRR-III-C9-2023-I8 grant CF 149/31.07.2023 {\em Conformal Aspects of Geometry and Dynamics}.   All three authors were partially supported by MATHAMSUD Regional Program 24-MATH-12. 

\section{{Preliminaries} }\label{sec:skt}

	Let $V$ be an $n$-dimensional real vector space equipped with an inner product $g$. For $p\geq 1$, the space of symmetric $p$-tensors on $V$, denoted by $\text{Sym}^{p}(V)$, is the vector subspace of the $p$-th tensor product $V^{\otimes p}$ of $V$ spanned by elements of the form 
	\[v_{1} \cdot \ldots \cdot v_{p}   :=   \sum_{\sigma \in \mathfrak{S}_{p} }   v_{\sigma(1)} \otimes \ldots \otimes v_{\sigma(p)} , \qquad \mbox{with }v_i\in V, \,i=1, \ldots, p.
	\] We may also denote $v\cdot \ldots\cdot v\in \Sym^p(V)$ simply by $v^{p}$. The inner product $g$ on $V$ can be extended to $\Sym^{p}(V)$ by setting: \begin{equation} \label{gextend}
		g(v_{1} \cdot \ldots \cdot v_{p},u_{1} \cdot \ldots \cdot u_{p}):= \sum_{\sigma \in \mathfrak{S}_{p} } g(v_{1},u_{\sigma(1)}) \ldots g(v_{p},u_{\sigma(p)}).
	\end{equation}
	 If $ \{ e_{i} \}_{1\leq i \leq n}$ is an orthonormal basis of $V$, the family $ {\{ e_{i_{1}} \cdot\ldots\cdot e_{i_{p}}| 1\leq i_{1}\leq \ldots \leq i_{p}\leq n\}}$ is an orthogonal (but not orthonormal) basis of $\Sym^{p}(V)$.

	Using the metric $g$, we identify $V$ and its dual $V^{*}$ and, more generally, $\Sym^{p}(V)$ with $\Sym^{p}(V^{*})$. Also, any symmetric $p$-linear map $T$ from $V^{p}$ to $\mathbb{R}$ can be identified with the following element in $\Sym^{p}(V)$ 
	\[
	\tfrac{1}{p!}\sum_{1\leq i_{1}\leq \ldots \leq i_{p}\leq n}T(e_{i_{1}},\ldots,e_{i_{p}})\, e_{i_{1}}\cdot\ldots\cdot e_{i_{p}}.
	\]
	Under this identification, $g$ corresponds to $\frac12L$, where $ L:=\sum_{1\leq i\leq n}e_{i}\cdot e_{i}\in \Sym^2(V)$.
 
	Let us denote by $\End(V)$ the set of endomorphism of $V$ and by $\End^{S}(V)$ and $\End^{A}(V)$ its subspaces of symmetric and skew-symmetric endomorphisms with respect to $g$. Then the linear map 
	\begin{equation}\label{eq:SM}
		\text{\normalfont{End}}(V) \rightarrow  \Sym^{2}(V), \quad 
		E \mapsto  S_{E} := \tfrac{1}{2} \sum_{1 \leq j \leq n} Ee_{j}  \cdot  e_{j},
	\end{equation} 
 is an orthogonal projector and has $\Enda(V)$ as kernel. In particular, denoting by $E^{*}$ the adjoint map of $E$ with respect to $g$, we have $S_{E} = S_{E^{*}}$, since $E-E^{*}$ is skew-symmetric. Also, when restricted to $\Ends(V)$, this map becomes an isomorphism whose inverse takes 
	$e_{i} \cdot e_{j}$ to the endomorphism  $ x\in V \mapsto g(x,e_{i})e_{j} + g(x,e_{j})e_{i}$, for all $i,j=1, \ldots, n$.
	
	Let us denote $\Sym^\bullet(V):=\sqcup_{p\geq 0}\Sym^p(V)$. Any $E\in \End(V)$ can be extended to a derivation of $\Sym^\bullet(V)$, preserving $\Sym^p(V)$ for all $p\geq 0$, as follows (we also denote by $E$ this extension): For $p=0$, we have $\Sym^0(V) = \mathbb{R}$ and $E$ is the zero map. For $p\geq 1$, take an homogeneous element  $ v_{1} \cdot \ldots \cdot v_{p} \in \Sym^p(V)$, and set
	\begin{equation} \label{eq:Mactdef}E(v_{1} \cdot \ldots \cdot v_{p}):= \sum_{1 \leq i \leq p}  E(v_{i}) \cdot v_{1} \cdot \ldots \cdot \widehat{v_{i}} \cdot \ldots \cdot v_{p},
	\end{equation} and extend it linearly to the whole space $\Sym^p(V)$. 
  	It is straightforward to check that  ${E(A\cdot B)  = E(A)\cdot B + A\cdot E(B) }$ for all $A,B\in \Sym^{\bullet}(V) $.

	The following properties will be useful and follow by direct computations: for all $E\in \End(V)$, $F \in \Ends(V)$,
	\begin{equation}\label{eq:NSM}
		E(g) =  2S_{E}, \qquad  E(S_{F}) = 2 S_{EF}.
	\end{equation}

	Let $(M,g)$ be a $n$-dimensional connected Riemannian manifold with Levi-Civita connection $\nabla$ and let $\{ e_{i} \}_{1\leq i \leq n}$ be a local orthonormal frame. A symmetric $p$-tensor on $M$ is a section of the bundle $\Sym^p(TM)$.
	\begin{definition}
		The Killing operator $\deri$ of $M$ is the linear map defined as \begin{equation}\label{dkilling}
			\begin{array}{crcl}
				\deri:& \Gamma(\text{\normalfont{Sym}}^{p} TM) &\rightarrow    & \Gamma(\text{\normalfont{Sym}}^{p+1} TM) \\
				& K &\mapsto & \deri(K):=\sum_{1 \leq i \leq n} e_{i} \cdot \nabla_{e_{i}} K.
			\end{array}
		\end{equation}
	\end{definition}		
	One can easily check that $\deri (K)$ does not depend on the choice of the local orthonormal frame. In addition,  $\deri$ satisfies the Leibniz rule: $\deri(R\cdot S) = \deri(R)\cdot S+ R\cdot \deri(S)$ for any $R, S$ symmetric tensors on $M$ \cite[Lemma 2.1]{HMS16}.
	
	With respect to this operator, we define the central notion of this paper, the symmetric Killing tensors on $(M,g)$.
	
	\begin{definition} \label{defk} A symmetric Killing $p$-tensor $K$ on $M$, with $p\geq 0$, is a section of the bundle ${\Sym^{p}(TM)}$ such that $\deri( K )= 0$. The set of all these tensors is a denoted by $ \Kil^{p}(M) $, and we further set ${\Kil(M) := \sqcup_{p\geq 0} \Kil^{p}(M)} $.
	\end{definition} 
	
	It is clear from \eqref{dkilling} that the Riemannian metric and, more generally, any parallel tensor on $M$ is a Killing tensor. Also, since $\deri$ satisfies the Leibniz rule,  every symmetric product of Killing tensors lies again in $\Kil(M)$. This motivates the following definition that is adapted from the ones given by \cite{dBM,Heiltes}.
	
	\begin{definition}\label{defdecomp}
			\begin{enumerate}
				\item 			 Let $(M,g)$ be a Riemannian manifold. The metric tensor and the Killing vector fields, i.e.~the elements of  the set  $\{g\} \cup \Kil^1(M)$, are called primitive Killing tensors of $M$.
				\item 				A symmetric Killing $p$-tensor $K \in \Kil^{p}(M)$, $p\geq 1$, is said to be decomposable if it is a polynomial in the primitive Killing tensors, i.e.~if $ K \in \mathbb{R} [g,\Kil^1(M)] $.
		\end{enumerate}		
	\end{definition}

		\begin{remark} \label{rmrk} In order to show that a symmetric Killing tensor $K$ has the decomposable expression $ P \in\mathbb{R} [g,\Kil^1(M)]$, it is enough to show that both tensors coincide in a non-empty open set of $M$. Indeed, it is well known that if a symmetric Killing tensor vanishes in a non-empty open set of a connected manifold, then it must vanish identically on $M$ \cite[Theorem 7]{matveev20}. Applying this result to $K-P$, proves the claim.
	\end{remark}

	\section{Left-invariant symmetric Killing tensors {on Lie groups}}\label{sec:leftinv}

	Along this section, $G$ is a connected Lie group endowed with a left-invariant Riemannian metric $g$. In addition, we denote with $\mg$ the Lie algebra of $G$, which carries an inner product, also denoted by $g$, induced by evaluating the Riemannian metric at the identity $e\in G$.  We fix an orthonormal basis $\{e_i\}_{i=1}^n$ of $\mg$.  Below, we study the Killing condition for left-invariant symmetric tensors on $G$. 
	
	The left and right translations on $G$ by an element $a\in G$ are denoted, respectively, by 
	\[\mathcal{L}_{a}: u \in G\mapsto au \in G, \mbox{ and } \mathcal{R}_{a}: u \in G\mapsto ua \in G.\] We may also consider the conjugation by $a$ defined as the composition $\mathcal{I}_a:= \mathcal{L}_a\circ \mathcal{R}_{a^{-1}} $ Clearly, each of these maps is a diffeomorphism. Left-invariance of $g$ is equivalent to $\mathcal{L}_a$ being an isometry of $(G,g)$  for all $a\in G$.
	
	\begin{definition}
		A $(k,l)$-tensor field $S$ defined on $G$ is said to be left-invariant if it satisfies $(\mathcal{L}_{a})_{*} S = S$, for all $a \in G$, where $(\mathcal{L}_{a})_{*}$ denotes the push-forward induced by $\mathcal{L}_a$.
	\end{definition} 
	
	With this notion, the space of left-invariant $(k,l)$-tensor fields on $G$ is identified with the space of $(k,l)$-linear tensors in $\mg$, namely, with elements of $\mg^{\otimes k} \otimes (\mg^{*})^{\otimes l} $. In particular, left-invariant symmetric $p$-tensors on $G$ are identified with $\text{\normalfont{Sym}}^{p}(\mg)$, and we denote by $\Kil^{p}(\mg)\subset \Sym^p(\mg)$ the subspace corresponding to left-invariant symmetric tensors on $G$ satisfying the Killing equation of the Definition \ref{defk}.
	
	We denote by $\nabla$ the Levi-Civita connection corresponding to $(G,g)$. 
	For left-invariant vector fields $X$ and $Y$ on $G$, $\nabla_{Y} X$ is again left-invariant; thus, through the above identification, $\nabla_{Y} X$ has a corresponding element in $\mg$, denoted by $\nabla_{y} x$, for $x,y$ the values of the vector fields $X,Y$ at the identity of $G$. Koszul's formula implies
	\begin{equation}
		\label{eq:koszul}
		\text{$\nabla_{y} x  =   \tfrac{1}{2} \Big( \ad_{y}x - \ad_{y}^{*}x - \ad_{x}^{*}y \Big)$, for all $x,y\in \mg$}.
	\end{equation}		
	One can easily show that for a left-invariant symmetric $p$-tensor $K$ on $G$, $\d(K) $ defined in \eqref{dkilling} is also left-invariant. Then, $\d(K)$ has a corresponding element in $\Sym^{p+1}(\mg) $ and we can consider the restriction of the map \eqref{dkilling} on left-invariant symmetric tensors as a map from $\Sym^{p}(\mg)$ to $ \Sym^{p+1}(\mg)$.

	In the context above, the following proposition gives an algebraic expression for the Killing operator when restricted to left-invariant tensor fields, and thus identified with the symmetric product in the Lie algebra.

		\begin{proposition} \label{pro:dKinv} For any $K\in \Sym^p(\mg)$, 
			\begin{equation}\label{eq:dK}
				\d (K)=\sum_{1\leq j \leq n} e_{j} \cdot\ad_{e_{j}}(K) ,
			\end{equation} where $\ad_{e_{j}}$ acts on $K$ as described in \eqref{eq:Mactdef}.
			
			\begin{proof}  We proceed by induction. For $p = 1$, Koszul's formula and \eqref{eq:SM} give
				\[
				\deri (x)=  \sum_{1 \leq i \leq n} e_{i} \cdot \nabla_{e_{i}} x  = -\tfrac{1}{2}\sum_{1 \leq i \leq n}  e_{i} \cdot \left( \ad_{x}e_{i} +\ad_{e_{i}}^{*}x + \ad_{x}^{*}e_{i} \right) =  \sum_{1 \leq i \leq n}  e_{i} \cdot \ad_{e_{i}}x 
				%
				%-S_{\ad_{x}} -  \tfrac{1}{2}\sum_{1 \leq i \leq n}  e_{i} \cdot \ad_{e_{i}}^{*}x - S_{\ad_{x}^{*}} 	=-2 S_{\ad_x},
				%		\sum_{1 \leq j \leq n} e_{j}\cdot \ad_{e_{j}} x &=&- \sum_{ 1 \leq i \leq n } \ad_{x} e_{i} \cdot e_{i}= -2S_{\ad_{x}}=\d(x),
				\] where the last equality holds since 
				$$\sum_{1 \leq i \leq n}e_i\cdot \ad_x^*e_i=\sum_{1 \leq i,j \leq n}e_i\cdot e_j \,g(\ad_x^*e_i,e_j)=\sum_{1 \leq i,j \leq n}e_i\cdot e_j \,g(e_i,\ad_xe_j)=\sum_{1 \leq j \leq n}\ad_xe_j \cdot e_j,$$ 
				 and the sum of the middle terms vanishes.
				
				Now, suppose \eqref{eq:dK} holds for symmetric tensors of degree $\leq p$. Then, applying the Leibniz rule for $\d$ and $\ad_{e_j}$ and the inductive hypothesis, we get for every $x\in\mg$ and $K\in \Sym^p(\mg)$:
				$$\sum_{1\leq j \leq n} e_{j} \cdot\ad_{e_{j}}(x\cdot K) =\sum_{1\leq j \leq n} e_{j} \cdot(\ad_{e_{j}}(x)\cdot K+x\cdot \ad_{e_{j}}(K))=\d(x)\cdot K+x\cdot \d(K) =\d(x\cdot K).$$ The formula thus holds for tensors in $\Sym^{p+1}(\mg)$ of the type $x\cdot K$ with $x\in\mg$ and $K\in \Sym^p(\mg)$, so by linearity it holds on $\Sym^{p+1}(\mg)$.
			\end{proof}
		\end{proposition}

		The previous proposition allows us to characterize symmetric Killing tensors of degrees 1 and 2. 
		
		\begin{corollary}   \label{cor:dK12} 
			\begin{enumerate}
				\item For every $x\in \mg$, $\d(x)=-2S_{\ad_x}$. 
				\item Given $K \in \Sym^2(\mg)$ a symmetric 2-tensor, one has
				\[K \in \Kil^{2}(\mg)  \mbox{ if and only if } \sum_{1\leq i \leq n} e_{i} \cdot S_{\ad_{e_{i}}\circ K}  = 0 ,\]
				where, on the right hand side, $K$ is viewed as a symmetric endomorphism of $\mg$.
			\end{enumerate}
		\end{corollary}	
		
		\begin{proof} The first assertion follows directly from \eqref{eq:SM} and \eqref{eq:dK} (see also \cite[Section 2.2]{dBM4}). For the second item, let $K\in \Sym^2(\mg)$. Proposition \ref{pro:dKinv} and \eqref{eq:NSM} imply 
			\[ \d(K)= \sum_{1\leq i \leq n}e_{i}\cdot \ad_{e_{i}}(K)= 2\sum_{1\leq i \leq n}e_{i}\cdot S_{\ad_{e_{i}}\circ K}, \] where $K$ is viewed as an element in $\End^S(\mg)$ in the last equation. This proves our claim.
		\end{proof}
		
		For every $(k,l)$-tensor field $S$ defined on $G$, we consider the map \begin{equation} \label{mapone}
			u \in G \mapsto (\mathcal{L}_{u^{-1}})_{*} S_{u} \in  \mg^{\otimes k} \otimes (\mg^{*})^{\otimes l}. 
		\end{equation}  In particular,  $S$ is left-invariant if and only if this map is constant. The composition of this map with the exponential map of $G$ defines the function
		\begin{equation} \label{mapomeg}
			\begin{array}{crcccl}
				\Omega_{S} : & \mg &\rightarrow & G & \rightarrow &  \mg^{\otimes k} \otimes (\mg^{*})^{\otimes l} \\
				& w & \mapsto &  \exp(w) & \mapsto &  (\mathcal{L}_{\exp(-w)})_{*} S_{\exp(w) } .
			\end{array}
		\end{equation} 
		In particular, for every symmetric product of vector fields, one has \begin{equation} \label{omegvf}
			\text{$\Omega_{X_{1}\cdot \ldots\cdot X_{k}}(w) = \Omega_{X_{1}}(w)\cdot \ldots\cdot \Omega_{X_{k}}(w)$, for all $w\in \mg$, $X_{i}$ vector fields on $G$.}
		\end{equation} 

     We denote by $\Iso(G,g)$ the isometry group of $(G,g)$, i.e., the set of diffeomorphisms $ f:G\rightarrow G$ satisfying $f^*g = g$. The isometry group has a natural Lie group structure and its Lie algebra is isomorphic to the Lie algebra of complete Killing vector fields on the manifold $G$. Since homogeneous manifolds are complete, $\Lie(\Iso(G,g))$ encodes all the information about Killing vector fields on our Riemannian manifold $(G,g)$.
	
	It is clear that $G$ is a Lie subgroup of $\Iso(G,g)$ by considering the map $a\in G\mapsto \mathcal{L}_a\in \Iso(G)$. At the Lie algebra level, this injection defines a map $\mg \mapsto \Kil^1(G) $ that can be described as follows: for each $x \in \mg$, we consider the vector field $\xi_{x}\in  \Kil^1(G)$ generated by right-translations of $x$. Namely, for any $u\in G$, $(\xi_{x})_{u} =  (\mathcal{R}_{u})_{*} x = (\mathcal{L}_{u})_{*} \text{Ad}(u^{-1})x$. Using this expression for  $u= \exp(w)\in \exp(\mg)$ we get an expression for the map $\Omega_{\xi_{x}} $ defined in \eqref{mapomeg}, which becomes 
 \begin{multline} \label{eqomg}
     \Omega_{\xi_{x}} (w) = \Ad(\exp(-w))x = e^{-\ad(w)}x%= \Big( \sum_{j\geq 0} \tfrac{(-1)^j}{j!} \ad_w^{j}  \Big) x \\
     =   x -[w,x]+ \tfrac{1}{2}[w,[w,x]] - \ldots \in \mg .
 \end{multline}
Note that the left-invariant vector field induced by $x\in\mg$ is Killing if and only if $\ad_x$ is a skew-symmetric map of $\mg$, due to Corollary \ref{cor:dK12}(1).
	
	Denote by $\Auto(G)$ the set of Lie group automorphisms of $G$ which are isometries with respect to $g$. It also has a Lie a group structure and, when $G$ is simply connected, it is isomorphic to $\Auto(\mg) $, the Lie group of automorphism of $\mg$  which are linear isometries. The Lie algebra of the latter is the following
	\[  
   \Dera(\mg):=	  \{ \text{$T\in \so(\mg)$: $T[x,y] = [Tx,y] + [x,Ty]$ for all $x,y$ $\in$ $\mg$}\} .
	\]	
	
 We now describe the procedure to induce Killing vector fields on $G$ from elements in $\Dera(\mg)$. For $T\in \Dera(\mg)$, $e^{tT} \in \Auto(\mg)$ for all $t\in \mathbb{R}$. If $G$ is simply connected, then for each $t$ there exists an isometry $f_t$  of $(G,g)$  such that $(f_t)_{*e} =   e^{tT} $. Therefore, ${(\xi_T)_u = \frac{d}{dt} \Bigr|_{t=0} f_t(u)}$ is a Killing vector field which we call {\em induced by $T$}; note that $f_t$ is the flow of $\xi_T$. If $G$ is not simply connected, and $G=\tilde G/\Gamma$ for $\tilde G$ its universal cover and $\Gamma$ a central discrete subgroup, the Killing vector field $\xi_T$ of $\tilde G$ descends to $G$ if and only if  $f_t(\gamma)=\gamma$  for all $t\in\R$ and $\gamma\in\Gamma$ (i.e.~$\xi_T$ vanishes at each point of $\Gamma$). In this case, we also denote by $\xi_T$ the induced Killing vector field on $G$ and by $f_t$ its flow.

 Let $\xi_T$ be a Killing vector field induced by $T\in \Dera(\mg)$ on $G$.   Considering $u\in G$ of the form $u = \exp(w)$, $w\in\mg$, canonical computations give % $f_t (\exp(w)) = \exp( (f_t)_{*e} w )$ and thus
	\[
	(\xi_T)_{\exp(w)} %= \frac{d}{dt} \Bigr|_{t=0} \exp( (f_t)_{*e} w ) = \frac{d}{dt} \Bigr|_{t=0} \exp( e^{tT} w ) 
	= (\exp)_{*w} (Tw).
	\]
	By \cite[Theorem 1.7, Ch. II]{Hel01}, $(\exp)_{*w} = (\mathcal{L}_{\exp(w)})_{*e}  \Big( \tfrac{1 - e^{-\ad_x}}{\ad_x} \Big)$. Using this expression, the map  \eqref{mapomeg} corresponding to the vector field $\xi_T $ becomes
	\begin{equation}\label{eq:OmegaT}
		\Omega_{T} (w) = T(w) - \tfrac{1}{2}\ad_w T(w)+\tfrac{1}{6}\ad^2_w T(w)-\ldots.
	\end{equation}	For further details in the above computations, we refer the reader to \cite[Section 2]{dBM}.

The Killing vector fields mentioned above take into account the algebraic structure of the Lie group $G$. We thus introduce the following definition.
\begin{definition}    \label{def:algkill}	
On a Lie group $(G,g)$ endowed with a left-invariant metric, Killing vector fields that are right- or left-invariant, induced by skew-symmetric derivations and linear combination of these, are called {\em algebraic Killing tensors}.
\end{definition}

It is worth mentioning that, in general, Lie groups with left-invariant metrics possess Killing vector fields which are not algebraic (see, for instance, \cite{GoWi88}). However, there exist classes of Lie groups for which $\xi_x$ and $\xi_T$, for $x\in\mg$ and $T\in \Dera(\mg)$, span the vector space of Killing fields such as, for instance, the nilpontent and simply connected ones \cite{Wo63}.

		Let $S$ be a  $(k,l)$-tensor field on $G$. If $S$ is left invariant, then $\Omega_S$ is constant in $\mg$. However, if $\Omega_S$ is constant, we only get the identity $S_{\exp(w)}=(\mathcal{L}_{\exp(w)})_* S_e$ for the image of the exponential map.
		The following result shows that if $S$ is a symmetric Killing $p$-tensor, the last identity is enough to conclude that $S$ is left-invariant. 
		
		\begin{proposition}\label{pro:leftinviff}
			On a connected Lie group $G$, a symmetric  Killing $p$-tensor $S$ is left invariant if and only if $\Omega_S$ is a constant function. 
		\end{proposition}
		\begin{proof}Let $S$ be a symmetric  Killing $p$-tensor.  As pointed out above, if $S$ is left-invariant then $\Omega_S$ is constant. For the converse, assume $\Omega_S$ is constant and set $K$ the left-invariant symmetric $p$-tensor whose value at $e$ is $\Omega_S(0)=S_e$. Therefore, if $V$ is an open neighborhood of $e$ and $U$ is open in $\mg$ such that $\exp:U\to V$ is a diffeomorphism, we have $S|_V=K|_V$. Indeed, for $x\in V$, $x=\exp w$ for some $w\in U$ and 
			\[
			S_x= (\mathcal{L}_{\exp w})_*\Omega_S(w)=(\mathcal{L}_{\exp w})_*S_e=K_x,
			\]where we used the definitions of $\Omega_S$ and $K$.
			
			Moreover, since $S$ is Killing in $G$, $K$ is Killing in $V$ and thus, by using left-translations (which are isometries), we get that $K$ is a Killing $p$-tensor in $G$.  By Remark \ref{rmrk}, $K=S$ and thus $S$ is left-invariant. 			 
		\end{proof}

		\section{Symmetric Killing tensors in the almost abelian context}\label{sec:almostab}
		
		In this section, we describe the left-invarint symmetric Killing tensors on connected almost abelian Lie groups, by a correspondence with their Lie algebras. 
		
		\begin{definition} \label{amstabel}
			A Lie algebra is called almost abelian if it has an abelian ideal of codimension 1. 
		\end{definition}

		Let $\mg$ be an $(n+1)$-dimensional almost abelian Lie algebra with an inner product $g$, and let $\mh$ an abelian ideal of $\mg$ with codimension 1. We fix $b\in\mg$ satisfying $g(b,b)=1$ and $b \perp \mh$. Since $\mh$ is an ideal, it is preserved by $\ad_b$, and we denote \begin{equation} \label{deriv}
			\begin{array}{crcl}
				D := \ad_b|_{\mh}:& \mh &\rightarrow & \mh \\
				& h &\mapsto &  D(h)=[b,h] .	
			\end{array}
		\end{equation} It is easy to show that the Lie algebra $\mg$ is isomorphic to the semidirect product of $\R b$ and $\mh$ via the representation $\rho:\R b\to {\rm Der}(\mh)$ such that $\rho(b)=D$. We denote this semidirect product as $\R b\ltimes_D\mh$ and we write ${\mg = \mathbb{R} b  {\ltimes_D} \mh}$ from now on; this is the presentation of the almost abelian Lie algebras with which we will be working.
		
		We now fix an orthonormal basis $\{h_{1},\ldots,h_{n}\}$ of $\mh$. Recall that, in Section \ref{sec:skt}, we set $L=2g$ which verifies $L\in \Kil(\mg)$. We define the symmetric 2-tensor ${L}_{\mh} := \sum_{1\leq i \leq n}h_{i}^{2}\in \Sym^2(\mh)$, and we thus get $L = b^{2} + L_{\mh}$. For the next result, we consider $D$ both as an endomorphism of $\mh$ and also as an endomorphism of $\mg$, extending it by zero on $b$.

		\begin{proposition}\label{properties}  \begin{enumerate}
				\item \label{properties1} $\deri(b)=-2S_D=-\tfrac{1}{2}D(L_{\mh})$. In particular, $b \in \Kil^{1}(\mg)$ if and only if $D$ is skew-symmetric.
				\item \label{properties2} For every $K \in \Sym^{p}(\mh)$, $\deri(K) = b \cdot D(K)$. In particular, $K \in \Kil^{p}(\mg)$ if and only if $D(K) = 0$.
			\end{enumerate}	
		\end{proposition}
		\begin{proof}
			By Corollary \ref{cor:dK12}(1), $\d (b)=-2S_{\ad_b}=-2S_D=-\sum_{1\leq j\leq n} D (h_j)\cdot h_j= -\frac12 D(L_\mh)$. 
			
			Now, let $K \in \Sym^{p}(\mh)$; by Proposition \ref{pro:dKinv},  $\deri(K) = b \cdot \ad_b(K)+ {\sum_{1\leq j\leq n}h_j\cdot \ad_{h_j} K}$. Since $K$ is a symmetric tensor in $\mh$, which is abelian, $\ad_{h_j} K=0$ for all $j$, so (2) follows. ~The rest of the proof is immediate since being Killing is equivalent to $\d(K)=0$.
		\end{proof}

		The following proposition provides a useful presentation of symmetric tensors on $\mg$, when interpreted as homogeneous polynomials in $b$. 
		
		\begin{lemma} \label{resulteuclid}
			For every $K \in \Sym^{p}(\mg)$, there exist unique symmetric tensors  \linebreak $\alpha_{i} \in \Sym^{p-2i-1}(\mh)$,  $\beta_{i} \in \Sym^{p-2i}(\mh)$, $i=1, \ldots, \lfloor \tfrac{p}{2} \rfloor$, such that 
			\begin{equation} 
				\label{presentation}
				K = \sum_{0 \leq i \leq   \lfloor \tfrac{p}{2} \rfloor  } L^{i} \cdot q_{i},		\end{equation}
			where $q_{i} := \alpha_{i}\cdot b + \beta_{i}$. By convention we set $\Sym^{-1}(\mh) := 0$.
			\begin{proof} 	We first show uniqueness. Assume that \eqref{presentation} holds.  Then, for every $s=1, \ldots, \lfloor \frac{p}2\rfloor$, %\linebreak 
				$\sum_{0 \leq i \leq   s-1  } L^{i} \cdot q_{i}$  is the reminder of the division of $K$ by $L^s$ in $\R[h_1, \ldots, h_n][b]$. So $q_0$ is unique and, since multiplication by $L$ is injective, all $q_i$ are uniquely determined. 
				
				To show existence, we proceed by induction. For $p = 0$,  $K\in \Sym^{0}(\mg)=\mathbb{R}$.  We have to write $K=L^0\cdot q_0$ with $q_0=\beta_0\in \R$, since $\alpha_0\in \Sym^{-1}(\mh)=0$. But $L^0=1$, so  $\beta_0:=K$ satisfies \eqref{presentation}. 
				
				For $p = 1$, $ K\in \Sym^{1}(\mg)=\mg$  and we have to write $K=L^0\cdot q_0=q_0$ with $q_0=\alpha_0 b+\beta_0$ with $\alpha_0\in \R$ and $\beta_0\in \Sym^1(\mg)$. Since $\mg=\R b\oplus \mh$ and $K\in \mg$, there exist $\alpha_0\in \R$ and $\beta_0\in \mh$ such that $K=\alpha_0 b+\beta_0$. 
				
				Assume that the statement holds for symmetric tensors of degree $p-2\geq0$. Let $ K \in \Sym^{p}(\mg)$. We apply the division algorithm in $\R[h_1, \ldots, h_n][b]$ to write ${K = L \cdot Q + q_{0}}$, where $Q \in \Sym^{p-2}(\mg)$ and  the remainder $q_{0}$ has degree  at most 1 in $b$. By the induction hypothesis, for $i=0,\ldots, {\lfloor\frac{p-2}2\rfloor}$ there exist $\alpha_i'\in \Sym(\mh)^{p-2i-3}, \beta_i'\in \Sym(\mh)^{p-2i-2}$  such that
				\[
				Q =  \sum_{0 \leq i \leq  \lfloor \tfrac{p-2}{2} \rfloor  }  L^{i} \cdot q'_{i}, 
				\]  where $q_i'=\alpha_i' b+\beta_i'$. Therefore, 
				\begin{multline*}
					K %=  L \cdot Q + q_{0}  
					=  L \cdot \Bigg( \sum_{0 \leq i \leq  \lfloor \tfrac{p-2}{2} \rfloor  }  L^{i} \cdot q'_{i} \Bigg)	+ q_{0}	   =  \sum_{0 \leq i \leq   \lfloor \tfrac{p}{2} \rfloor -1  }  L^{i+1} \cdot q'_{i} + q_{0}   \\
					=  \sum_{1 \leq j \leq   \lfloor \tfrac{p}{2} \rfloor   } L^{j} \cdot q_{j-1}' + L^{0} \cdot q_{0} =\sum_{0 \leq j \leq   \lfloor \tfrac{p}{2} \rfloor  } L^{j} \cdot q_{j} ,
				\end{multline*}
				where $q_{j} :=  q_{j-1}' $ for $ 1 \leq j \leq  \lfloor \tfrac{p}{2} \rfloor. $  Dividing $q_0$ by $b$ we get $\alpha_0\in \R, \beta_0\in \mh$ such that $q_0=\alpha_0 b+\beta_0$; in addition we take $\alpha_i:=\alpha_{i-1}'$ and $\beta_i:=\beta_{i-1}'$ for $ 1 \leq i \leq  \lfloor \tfrac{p}{2} \rfloor$, and the result follows. 
			\end{proof}
		\end{lemma}

		Consider $K \in \Sym^p(\mg) $ and  write it as in Lemma \ref{resulteuclid}.  From now on, we denote $r:= \lfloor \tfrac{p}{2} \rfloor  $ for convenience.  Then, we can write $K$ as follows
		
		\begin{equation}\label{decomp}
			K= \sum_{0 \leq i \leq r}L^{i}\cdot(\alpha_{i}\cdot b + \beta_{i}) =K^{\text{\normalfont{o}}}+K^{\text{\normalfont{e}}}, \mbox{ where }
			\Ko :=b\cdot   \sum_{0 \leq i \leq r} L^{i}\cdot  \alpha_{i},\;		 \Ke:= \sum_{0 \leq i \leq r} L^{i}\cdot  \beta_{i} .
		\end{equation}
		Notice that $K^{\text{\normalfont{o}}}$  and $K^{\text{\normalfont{e}}}$  have odd and even degree in $b$, respectively.

		\begin{proposition} \label{pro:KoKe} In the notation above, $K$ is Killing if and only if $\Ko$ and $\Ke$ are Killing.
			\begin{proof} Considering $K$ as in \eqref{decomp}, and recalling that $\deri(L)=0$, we use Proposition \ref{properties} to compute $\deri(K)$: 
   \begin{eqnarray*}
					\deri(K)& = &  \deri(\Ko) + \deri(\Ke)     \\
					& = & \Bigg(-2S_D \cdot \sum_{0 \leq i \leq r} L^{i}\cdot  \alpha_{i} + b\cdot \sum_{0 \leq i \leq r} L^{i}\cdot b \cdot D( \alpha_{i} ) \Bigg) + \Bigg(\sum_{0 \leq i \leq r} L^{i}\cdot b \cdot D( \beta_{i} )  \Bigg).
				\end{eqnarray*}   Since $S_D=\tfrac{1}{4}D(L_{\mh})\in \Sym^2(\mh)$, we have that $\deri(\Ko) $  has even degree in $b$, while $\deri(\Ke)$ only odd degree in $b$. Therefore, $\deri(K)=0$ if and only if $ \deri(\Ko) = 0$ and $\deri(\Ke) = 0$.	
			\end{proof}
		\end{proposition}

		\begin{theorem} \label{biggestthr}Let $K = \Ko + \Ke\in \Sym^{p}(\mg)$ as above. Then \begin{enumerate}
				\item $\Ke$ is Killing if and only if $\beta_{i}$ is Killing for all $i=0, \ldots, r$.
				\item If $D$ is skew-symmetric, $\Ko$ is Killing if and only if $\alpha_{i}$ is Killing for all $i=0,\ldots, r$.
				\item If $D$ is not skew-symmetric, $\Ko$ is Killing if and only if $\Ko =0$.
			\end{enumerate}	
			\begin{proof} (1) We have $\textstyle \deri(\Ke) = \sum_{0 \leq i \leq r} L^{i}\cdot b \cdot D( \beta_{i} )$, and by the uniqueness guaranteed by Lemma \ref{resulteuclid} (since $ \lfloor \tfrac{p}{2} \rfloor \leq \lfloor \tfrac{p+1}{2} \rfloor  $), we have
\[ \deri(\Ke) = \sum_{0 \leq i \leq r} L^{i}\cdot b \cdot D( \beta_{i} ) = 0 \Leftrightarrow D(\beta_{i}) = 0\ \forall i=0,\ldots, r . \]

(2) If $D$ is skew-symmetric we have that $\deri(b)=0$ by Proposition \ref{properties}\eqref{properties1}, so $\textstyle \deri(\Ko) = \sum_{0 \leq i \leq r} L^{i}\cdot b^2 \cdot D( \alpha_{i} )$, and again $\Ko$ is Killing if and only if $D(\alpha_i)=0$ for every $i=0,\ldots, r$.

(3) Assume now that $D$ is not skew-symmetric, so $S_D$ is a non-zero element in $\Sym^2(\mh)$. We have the following expression for $\deri(\Ko) $: 	
				\begin{eqnarray*} 
					\deri(\Ko)	& = & \sum_{0 \leq i \leq r} L^{i} \cdot \Big(  b^{2}\cdot D(\alpha_{i} ) -2 \alpha_{i} \cdot S_D \Big) \\
					& = & \sum_{0 \leq i \leq r} L^{i} \cdot \Big(  (L-L_{\mh})\cdot D(\alpha_{i} )  -   2\alpha_{i} \cdot S_D   \Big)          \\
					& = &   \sum_{0 \leq i \leq r}  L^{i+1}  \cdot D(\alpha_{i} ) - \sum_{0 \leq i \leq r}   L^{i} \cdot L_{\mh} \cdot D(\alpha_{i} )  -2   \sum_{0 \leq i \leq r} L^{i} \cdot \alpha_{i} \cdot  S_D. 
				\end{eqnarray*} Therefore, 
\begin{eqnarray*}		
\deri(\Ko) & = & L^{r+1} \cdot D(\alpha_{r})  +  \sum_{1 \leq i \leq r} L^{i} \cdot \Big( D(\alpha_{i-1} ) -  L_{\mh} \cdot D(\alpha_{i} )-2\alpha_{i} \cdot  S_D \Big)  
\\					&& + \;  L^{0} \cdot \Bigg( -2\alpha_{0} \cdot S_D -  L_{\mh} \cdot D(\alpha_{0} )    \Bigg)  . 
				\end{eqnarray*}		
				
				Notice that $D(\alpha_{r}) = 0$, since either $p$ is even and $\alpha_{r}\in\Sym^{-1}(\mh) = 0$ or $p$ is odd and $\alpha_{r}\in\Sym^{0}(\mh) = \mathbb{R}$.
			
			Therefore \begin{eqnarray*}
					\deri(\Ko)& = &  \sum_{1 \leq i \leq r} L^{i} \cdot \Big( D(\alpha_{i-1} ) -  L_{\mh} \cdot D(\alpha_{i} ) -2 \alpha_{i} \cdot S_D \Big)  \\
					&  &+  \;L^{0} \cdot\Bigg( -2\alpha_{0} \cdot S_D -  L_{\mh} \cdot D(\alpha_{0} )    \Bigg)  	.
				\end{eqnarray*} Again, by the uniqueness of the decomposition, $\deri(\Ko) = 0$ implies that each coefficient of $L^{i}$ must be zero. Then, the above equality is equivalent to the system
				\begin{equation} \label{system}
					\left\{\begin{array}{rcl}
						0&=&2S_D \cdot \alpha_{0}  +  D(\alpha_{0} ) \cdot  L_{\mh}  \\
						D(\alpha_{i} ) &=& 2  S_D  \cdot  \alpha_{i+1}+  D(\alpha_{i+1}) \cdot  L_{\mh} , \quad i=0, \ldots, r-1.
					\end{array}\right.
				\end{equation}	
				
				We claim that the system \eqref{system} is equivalent to 
				\begin{equation} \label{system2}
					\left\{\begin{array}{rcl}
						\alpha_0 &=& \sum_{i=1}^r (-1)\alpha_i \cdot L_\mh^i\\
						D(\alpha_k)&=&2S_D \cdot \sum_{i=k+1}^r\alpha_i \cdot  L_\mh^{i-k-1}, \quad k=0, \ldots, r.
					\end{array}\right.
				\end{equation}
				In fact, a simple induction argument  using the second  equation in \eqref{system} gives the second equation in \eqref{system2}. In particular, $D(\alpha_0)=2S_D  \cdot \sum_{i=1}^r\alpha_i  \cdot L_\mh^{i-1}$ and using this in the first equation of \eqref{system} gives
				\[
				0=2S_D \cdot \alpha_0+2S_D  \cdot \Bigg(\sum_{i=1}^r\alpha_i \cdot  L_\mh^{i-1} \Bigg) \cdot L_\mh=2S_D  \cdot  \Bigg(\sum_{i=0}^r\alpha_i  \cdot L_\mh^{i-1}\Bigg),
				\]which implies the first equation of \eqref{system2} since $S_D\neq 0$. Finally, one can easily check that \eqref{system2} implies \eqref{system}, thus proving our claim. 
    
				Next we show that \eqref{system2} implies
				\begin{equation}
					\label{eq:alphak}
					\alpha_k=\sum_{i=k+1}^r c_k(i)\alpha_i  \cdot L_\mh^{i-k} \quad \mbox{ with } c_k(i)<0, \mbox{ for all }  k=0, \ldots, r,\, i=k+1, \ldots r.
				\end{equation}
				
				We proceed by induction. For $k=0$ this holds due to the first equation in \eqref{system2}. Assume that \eqref{eq:alphak} is valid for some $k\le r-1$. Then, applying $D$ to \eqref{eq:alphak} and using the fact that $D$ acts as a derivation on tensors, we get
				\begin{eqnarray*}
					D(\alpha_k) &=&\sum_{i=k+1}^r c_k(i)D(\alpha_i)  \cdot L_\mh^{i-k}+\sum_{i=k+1}^r c_k(i)\alpha_i  \cdot D(L_\mh^{i-k})\\
					&=&2S_D \cdot \left(\sum_{i=k+1}^r 
					\sum_{j=i+1}^r c_k(i) \alpha_j  \cdot L_\mh^{j-k-1}
					+2\sum_{i=k+1}^r c_k(i) (i-k)\alpha_i \cdot L_\mh^{i-k -1}\right)\\
					&=&2S_D\cdot\left(\sum_{j=k+2}^r 
					\sum_{i=k+1}^{j-1} c_k(i) \alpha_j  \cdot L_\mh^{j-k-1}
					+2\sum_{i=k+1}^r c_k(i) (i-k) \alpha_i\cdot L_\mh^{i-k -1    } \right)\\
					&=&2S_D \cdot \left(\sum_{j=k+2}^r M_j(k)
					\alpha_j  \cdot L_\mh^{j-k-1}
					+2\sum_{j=k+1}^r c_k(j)(j-k)\alpha_j  \cdot L_\mh^{j-k -1}\right),
				\end{eqnarray*}where $M_j(k)=\sum_{i=k+1}^{j-1} c_k(i)$, $j=k+2, \ldots r$ and, in the second equality, we use the second equation in \eqref{system2} and $ D(L_\mh^{i-k})=4(i-k)S_D \cdot L_\mh^{i-k-1}$ (which follows from a simple induction).
				
				On the other hand, $D(\alpha_k)=2S_D \cdot  \sum_{j=k+1}^r\alpha_j  \cdot L_\mh^{i-k-1}$ by \eqref{system2}, so we obtain
				\begin{equation}
					\sum_{j=k+2}^r M_j(k)
					\alpha_j  \cdot L_\mh^{j-k-1}
					+2\sum_{j=k+1}^r c_k(j) (j-k)\alpha_j  \cdot L_\mh^{i-k -1}= \sum_{j=k+1}^r\alpha_j \cdot L_\mh^{j-k-1},
				\end{equation}
				and, since $c_k(k+1)<0$,  we can isolate $\alpha_{k+1}$ as
				\begin{equation}
					\alpha_{k+1}  = \sum_{j=k+2}^r\frac{1-  M_j(k)-2 c_k(j)(j-k)}{ ( 2c_k(k+1)-1)}\;\alpha_j \cdot L_\mh^{j-k-1}.
				\end{equation}
				Therefore, $\alpha_{k+1}=\sum_{j=k+2}^r c_{k+1}(j)\alpha_j L_\mh^{j-k}$ where 
				\[
				c_{k+1}(j)=\frac{1-  M_j(k)-2 c_k(j)(j-k)}{ (2c_k(k+1)-1)}, \quad j=k+2, \ldots, r.
				\] 
				By induction hypothesis $c_k(j)<0$ for all $i=k+1, \ldots, r$, so $2c_k(k+1)-1<0$,  $M_j(k)=\sum_{i=k+1}^{j-1} c_k(i)<0$ so  $1-  M_j(k)-2 c_k(j)(j-k))>0$ for all $j=k+2, \ldots, r$. Hence $c_{k+1}(j)<0$ for all $j=k+2, \ldots,r$ as we wanted to show. 
				
				Finally, it remains to prove that \eqref{eq:alphak} implies $\alpha_k=0$ for all $k=0, \ldots, r$. This follows immediately by using recursively that equation, indeed,  \eqref{eq:alphak} for $k=r$ implies $\alpha_r=0$. The same equation for $k=r-1$ gives $\alpha_{r-1}=c_{r-1}(r)\alpha_r \cdot L_\mh=0$, and so on. Therefore $\alpha_k=0$ for all $k=0,\ldots, r$ which gives $K^\circ=0$. \end{proof}
		\end{theorem}

      In the next example, we describe the decomposition in Lemma \ref{resulteuclid} and the characterization of the Killing condition in Theorem \ref{biggestthr} for the particular case $p=2$.
		
		\begin{example} Consider $K$ an element in $\Sym^{2}(\mg)$ written in the orthonormal basis $\{b,h_1, \ldots, h_n\}$ of $\mg$ as
			\[
			K=\lambda b^2+v\cdot b+ \sum_{1\leq i\leq j\leq n}a_{ij}h_i\cdot h_j,
			\] for some $\lambda,a_{ij}\in \R$, $i,j=1, \ldots, n$ and $v\in \mh$.
			Setting $\beta_0:=\sum_{1\leq i\leq j\leq n}(a_{ij}-\lambda \delta_{ij})h_i\cdot h_j$, $\alpha_0:=v$ and $\beta_1:=\lambda$ we get the decomposition as in  Lemma \ref{resulteuclid}:
			\begin{equation}\label{expressK}
				K =   L\cdot \beta_{1} + \alpha_{0}\cdot b + \beta_{0} .
			\end{equation} 			
			
			In this case, the odd and even parts of $K$ are, respectively,  $\Ko=\alpha_{0}\cdot b$ and $\Ke= L\cdot \beta_{1} +  \beta_{0} $. By Proposition \ref{pro:KoKe} and Theorem \ref{biggestthr}, $K$ is Killing if and only if $\beta_{1} $, $\alpha_{0}$ and $\beta_{0}$ are Killing. Since these are symmetric tensors in $\mh$,  the Killing condition is equivalent to being annihilated by $D$ due to Corollary \ref{properties2}(2). That is:
			\begin{itemize}
				\item $\beta_{1}\in \mathbb{R}$, then $D(\beta_{1}) = 0$ automatically, as mentioned in Section \ref{sec:skt};
				\item $\alpha_0\in\mh$, then  $D(\alpha_{0}) = 0$ if and only if $\alpha_{0} \in \Ker(D)$;		 
				\item  Viewing $\beta_{0}$ as a symmetric endomorphism of $\mh$, by  \eqref{eq:NSM} we have $D(\beta_{0}) = 2S_{D\circ \beta_{0}}$. Then $D(\beta_{0}) = 0 $ if and only if $D\circ \beta_{0}$ is a skew-symmetric endomorphism of $\mh$. 
			\end{itemize}	Therefore, $K$ in \eqref{expressK} is a symmetric Killing 2-tensor in $\mg$ if and only if $\alpha_0\in \Ker(D)$ and $D\circ \beta_0$ is skew-symmetric. 			
		\end{example}

		\section{Decomposability}\label{sec:decomp}

		We maintain the notation of the previous section so that $\mg$ denotes an $(n+1)$-dimensional almost abelian Lie algebra, which is isomorphic to $\mg = \mathbb{R} b  {\ltimes_D} \mh$, where $\mh$ is an abelian ideal and $D={\ad_b}|_\mh$. We fix an orthonormal basis $\{h_{1}, \ldots, h_{n} \}$  of $\mh$  with respect to the inner product $g$, and let $G$ be a connected Lie group with Lie algebra $\mg$. 
		
		In this section, we address the problem of decomposability of  left-invariant symmetric Killing tensors defined on $G$, in terms of Definition \ref{defdecomp}. Recall that left-invariant symmetric $p$-tensors on $G$ are identified with elements on $\Sym^p(\mg)$.

		We start by describing the Killing vector fields in $G$ via their associated function defined in \eqref{mapomeg}. For each $h_{i}$ in the fixed basis, we compute $\Omega_{\xi_{h_{i}}}$, where $\xi_{h_i}$ is the right-invariant vector field induced by $h_i$ as in  \eqref{eqomg}. For $w = \gamma b + h \in \mg$, $\gamma\in \mathbb{R}$ and $h \in \mh$, one has ${\ad_{w}^{j}(h_{i}) = \gamma^{j} D^{j} (h_{i})}$ for all $j\geq 0$, $i=1, \ldots, n$. Indeed: \[\ad_{w}(h_{i}) = [ \gamma b + h,h_{i}] = \gamma D(h_{i}) \]
		and, assuming the equality holds for $j$, we have \[ \ad_{w}^{j+1}(h_{i}) = [w, \ad_{w}^{j}(h_{i})] =  [ \gamma b + h,\gamma^{j} D^{j}(h_{i})] = \gamma^{j+1}D^{j+1}(h_{i}).  \] Using this expression in  \eqref{eqomg}, we get for all $i=1, \ldots, n$,
        \begin{equation}\label{Omegahi}
		\Omega_{\xi_{h_{i}}}(\gamma b + h) = e^{-\ad_w} (h_{i}) = \sum_{k\geq0}\tfrac{(-\gamma)^k}{k!}D^k(h_i) = e^{-\gamma D }(h_{i}).    \end{equation} Similarly, $\ad_{w}^n(b) =  -\gamma^{n-1} D^n(h)$ and then
        \begin{equation} \label{Omegab}
	\Omega_{\xi_{b}} (\gamma b + h) =  b + \sum_{k\geq1}\tfrac{(-\gamma)^{k-1}}{k!}D^k(h).\end{equation} 
 
    	We will now investigate skew-symmetric derivations of $\mg$. Given $T \in \so(\mg)$, we define $v:= T(b)$ and $T^{\mh} = \pr_{\mh} T|_\mh $, the projection to $\mh$ of the restriction $T|_\mh$ of $T$ to $\mh$. Then $g( b,T(b)  ) = 0$ implies $v\in\mh$ and we can write
	  \begin{equation}\label{eq:Tbh}
	      T (\gamma b+h)= \gamma T(b)+ T(h) = \gamma v + g( b,T(h) ) b + T^{\mh}(h)  =  \gamma v - g( v, h ) b  +  T^{\mh}(h) .
	  \end{equation}   
   Using the notation
   \[(v\wedge w)(x):=g(v,x)w-g(w,x)v , \quad \forall v,w,x\in \mg,\]the above relation reads $T = b\wedge v + T^\mh$.
    Straightforward computations show that  $T$ is a derivation of the Lie algebra $\mg$ if and only if  
	\begin{equation}\label{Tderivati}
		\left\{  \begin{array}{l}			
			\left[ T^{\mh} , D \right] =0\\
   \Im D\subset v^\bot\cap \mh\\
         |v|^2 \,D(v^\bot)=0
		\end{array} \right. 
	\end{equation} Notice that $D^2\neq 0$ implies $v =0$ due to the last two equations.
 
In addition, following and inductive argument, one can show 
	 \begin{equation}\label{adnT}
	 	\ad_w^{k} T(w) =  D^k( \gamma^{k+1} v + \gamma^k T^{\mh}(h) + \gamma^{k-1} g( v, h ) h   ) , \quad \forall w = \gamma b + h \in \mg,\;n\geq 1.
	 \end{equation}
Using this formula, the function $\Omega_{T}$ defined in \eqref{eq:OmegaT} has the following expression:
	\begin{multline}\label{OmegaTt}
		\Omega_{T}(\gamma b + h)  =  \gamma v +  T^{\mh}(h) - g( v, h ) b  -\tfrac{1}{2} D(\gamma^2v + \gamma T^{\mh}(h) +  g( v, h ) h ) \\  + \sum_{k\geq 2}\tfrac{(-1)^k}{(k+1)!}D^k(\gamma^{k+1} v + \gamma^k T^{\mh}(h) + \gamma^{k-1} g( v, h ) h  )  .
	\end{multline}

		In terms of Lie group 
		actions, $\GL(\mg)$ acts on $\Sym^{p}(\mg)$ as follows: for $A\in \GL(\mg)$ and $v_{1}\cdot\ldots\cdot v_{p}\in \Sym^{p}(\mg)$,  \begin{equation}\label{action}
			A(v_{1}\cdot\ldots\cdot v_{p}) = A(v_{1})\cdot\ldots\cdot A(v_{p}).  
		\end{equation} 
		
		We claim that for any  $D\in \gl(\mg)$,  $K\in \Sym^{p}(\mg)$ and $\gamma\in \mathbb{R}$, the action of $e^{-\gamma D}\in \GL(\mg)$ on $K$ satisfies
		\[		e^{-\gamma D}(K)   = \sum_{j\geq 0} \tfrac{(-\gamma)^j  }{j!} D^j(K)   =    K - \gamma D(K) + \tfrac{\gamma^{2}}{2}D^{2}(K)- \ldots,     \]where $D(K)$ is the action of $D$ on $K$ given in \eqref{eq:Mactdef}.
		
		Indeed, the action of $e^{-\gamma D}\in \GL(\mg)$ on a symmetric 2-tensor $v_{1}\cdot v_{2}\in \Sym^{2}(\mg)$ verifies \begin{eqnarray*}
			e^{-\gamma D}(v_{1}\cdot v_{2})	& \overset{\eqref{action}}{=} &e^{-\gamma D}(v_{1})\cdot e^{-\gamma D} (v_{2})     \\
			& = &  \Big(\sum_{i\geq 0} \tfrac{(-1)^{i}}{i!}\gamma^{i}D^{i}(v_{1})\Big) \cdot \Big( \sum_{j\geq 0} \tfrac{(-1)^{j}}{j!}\gamma^{j}D^{j}(v_{2}) \Big)  \\
			& = & \sum_{t\geq 0}(-1)^{t}\gamma^{t} \sum_{\substack{i,j\geq 0 \\ i+j = t}} \tfrac{1}{i!j!}D^{i}(v_{1})\cdot D^{j}(v_{2}) \\
			& = & \sum_{t\geq 0}(-1)^{t}\gamma^{t} \tfrac{1}{t!}D^{t}(v_{1}\cdot v_{2})    \\
			& = &   v_{1}\cdot v_{2} - \gamma D(v_{1}\cdot v_{2}) + \tfrac{\gamma^{2}}{2}D^{2}(v_{1}\cdot v_{2})- \ldots   .
		\end{eqnarray*} Our claim follows from an inductive argument for symmetric $p$-tensors. 
		
		The following result shows that any left-invariant symmetric Killing tensor $K$ on $\mg$  that verifies $K\in \Sym^p(\mh)$ is decomposable. More precisely:

		\begin{proposition} \label{propdecomp} Let $K \in \Sym^{p}(\mh) \cap\Kil^{p}(\mg)$ written as above. Then, $K$ can be written as a polynomial in right-invariant vector fields of $G$. 
			\begin{proof} Let $K$ be a symmetric Killing $p-$tensor on $\mg$ that verifies $K\in \Sym^p(\mh)$. For any such tensor, we have
		\[K =  \tfrac{1}{p!} \sum_{1 \leq i_{1}, \ldots, i_{p} \leq n} K(h_{i_{1}}, \ldots  h_{i_{p}}   )  h_{i_{1}} \cdot \ldots \cdot h_{i_{p}}.\]
		
		Using the components of $K$ above, we define the following symmetric Killing $p$-tensor on $G$
		\[S:=
		{ \tfrac{1}{p!} \sum_{1 \leq i_{1}, \ldots, i_{p} \leq n} K(h_{i_{1}}, \ldots  h_{i_{p}}   )  \xi_{h_{i_{1}}} \cdot \ldots  \cdot \xi_{h_{i_{p}}} },
		\]where $\xi_{h_i}$ is the right-invariant Killing vector field induced by $h_i$. Notice that, in principle, $S$ is not necessarily left-invariant. Next we show that this is in fact the case.

   We consider the function $\Omega_S$ defined in \eqref{mapomeg}, and we evaluate it on an arbitrary element $\gamma b + h$ of $\mg$, where $\gamma\in \R$ and $h\in\mh$. We obtain:
			\begin{eqnarray*}			    
				\Omega_S(\gamma b+h)&=& \tfrac{1}{p!} \sum_{1 \leq i_{1}, \ldots, i_{p} \leq n} K(h_{i_{1}}, \ldots  h_{i_{p}}   )  \Omega_{ \xi_{h_{i_{1}}} \cdot \ldots  \cdot \xi_{h_{i_{p}}}}( \gamma b + h) 
				\\ &=& \tfrac{1}{p!} \sum_{1 \leq i_{1}, \ldots, i_{p} \leq n} K(h_{i_{1}}, \ldots  h_{i_{p}}   ) \Omega_{ \xi_{h_{i_{1}}}}( \gamma b + h) \cdot \ldots  \cdot \Omega_{\xi_{h_{i_{p}}}}( \gamma b + h) 
				\\    &=& \tfrac{1}{p!} \sum_{1 \leq i_{1}, \ldots, i_{p} \leq n} K(h_{i_{1}}, \ldots  h_{i_{p}}   )   e^{-\gamma D}( h_{i_{1}} \cdot \ldots \cdot h_{i_{p}}  ) 
				\\ &=& e^{-\gamma D} \Bigg( \tfrac{1}{p!} \sum_{1 \leq i_{1}, \ldots, i_{p} \leq n} K(h_{i_{1}}, \ldots  h_{i_{p}}   )  h_{i_{1}} \cdot \ldots \cdot h_{i_{p}}   \Bigg)\\
				  &=& e^{-\gamma D}(K) = K,    
    			\end{eqnarray*}
			where the last equality holds since $K \in \Sym^{p}(\mh)$ is Killing in $\mg$ and thus $D(K)=0$ by Proposition \ref{properties}\eqref{properties2}.

			This shows that $\Omega_S$ is constant in $\mg$ and equal to the the symmetric $p$-tensor $K$.
			Therefore, by Proposition \ref{pro:leftinviff}, the tensor field $S$ on $G$ is left-invariant. Since it coincides with $K$ at $e$, we obtain that $S$ is equal to the (left-invariant) tensor field $K$ on $G$, and thus $K$ is a polynomial in right-invariant vector fields in $G$ as claimed. 
		\end{proof}
	\end{proposition}
	
	We now prove our main result: every left-invariant symmetric Killing tensor on an almost abelian Lie group is decomposable.

	\begin{theorem} \label{teo:main} Let $G$ be an almost abelian Lie group, equipped with a left-invariant Riemannian metric. Then, every left-invariant symmetric Killing $p$-tensor on $G$ is a polynomial in the metric and right- and left-invariant Killing vector fields.
	\end{theorem}

	\begin{proof} Denote by $\mg$ the Lie algebra of $G$. Consider $K\in\Kil^{p}(\mg)$ written $K=\Ko + \Ke$ as in \eqref{decomp}:
		\[K= K^{\text{\normalfont{o}}}+K^{\text{\normalfont{e}}}, \mbox{ where }
		\Ko :=b\cdot   \sum_{0 \leq i \leq \lfloor \tfrac{p}{2} \rfloor  } L^{i}\cdot  \alpha_{i},\;		 \Ke:= \sum_{0 \leq i \leq \lfloor \tfrac{p}{2} \rfloor  } L^{i}\cdot  \beta_{i} ,\] with 
		$\alpha_i\in \Sym^{p-2i-1}(\mh)$ and $\beta_i\in \Sym^{p-2i}(\mh)$.
		Since $K$ is Killing, $\alpha_i$ and $\beta_{i}$ are Killing for every $0\leq i \leq  \lfloor \tfrac{p}{2} \rfloor $,  
		by Proposition \ref{pro:KoKe} and Theorem \ref{biggestthr}.
		
		By Proposition \ref{propdecomp}, for each $i=0, \ldots \lfloor\frac{p}2\rfloor$, $\alpha_i$ and $\beta_i$ can be written as a linear combination of symmetric products of Killing vector fields. 
		Since $L = 2g \in \Kil(\mg)$ we get directly that $\Ke$ is decomposable. We claim that $\Ko$ is decomposable as well. 
		
		Indeed, when $D$ is not skew-symmetric then $\Ko=0$ by Theorem \ref{biggestthr}, so the claim holds trivially. For $D$ skew-symmetric, the left-invariant vector field determined by $b$ is a Killing vector field by Proposition \ref{properties}, and thus $\Ko$ is a polynomial in $L$ and Killing vector fields, i.e. it is decomposable in terms of Definition \ref{defdecomp}.
	\end{proof}

\section{Constant sectional curvature}\label{sec:spaceforms}
In this section, we focus on connected almost abelian Lie groups endowed with  left-invariant metrics that have constant sectional curvature. Almost abelian Lie groups are in particular solvable, so the Riemannian manifolds under consideration are either flat or negatively curved space forms. 

The motivation to study this particular case is the well known result of Thompson and Takeuchi \cite{takekilli,Thomps}: on a connected Riemannian manifold of constant sectional curvature, any symmetric Killing $p$-tensor can be written as a polynomial in Killing vector fields. 

Our aim is to analyze whether it is possible to refine Thompson and Takeuchi's result for almost abelian Lie groups with constant sectional curvature, in terms of their algebraic structure, i.e. using their algebraic Killing vector fields (see Definition \ref{def:algkill}).  More precisely, we ask the following:
\begin{question}
\label{q:question}Let $(G,g)$ be a connected almost abelian Lie group with a left-invariant metric of constant sectional curvature. 
Is every left-invariant Killing $p$-tensor on $G$ a polynomial in algebraic Killing vector fields? 
\end{question}

Note that if the answer is affirmative for a connected Lie group, then it is so also  for its universal cover. 

Let $(G,g)$ be a connected almost abelian Lie group endowed with a left-invariant Riemannian metric, and let $\mg=\R b\ltimes_D \mh$ be its Lie algebra, where $D=\ad_b|_{\mh}$. We maintain the notation of the previous sections. Then, $(G,g)$ has constant sectional curvature if and only if $D=\lambda \Id+A$, where $A$ is a skew-symmetric endomorphism of $\mh$ (see \cite[Theorem 1.5]{Mil76} and \cite[Theorem 4.2]{thompmetricsliegr}). Moreover, it is flat if and only if $\lambda=0$.

\begin{proposition}\label{pro:flat}
    If $(G,g)$ is flat then every left-invariant symmetric Killing $p$-tensor can be written as a polynomial in left- and right-invariant Killing vector fields.
\end{proposition}
\begin{proof}
If $(G,g)$ is flat, then $D$ is skew-symmetric by the results cited above. Therefore, the left-invariant vector field on $G$ induced by $b$ is Killing due to Corollary \ref{cor:dK12}(1) and $L_\mh=b^2-L$  is also a symmetric Killing tensor. 
		This implies, by Proposition \ref{propdecomp}, that $L_{\mh}$ is a polynomial in right-invariant Killing vector fields. From Theorem \ref{teo:main}, we obtain that every left-invariant symmetric Killing tensor $K$ is a polynomial in left- and right-invariant Killing vector fields. 
\end{proof}

This proposition answers by the affirmative Question \ref{q:question}. For constant non-zero scalar curvature, the situation is different as we shall see next. First of all, there are no non-trivial left-invariant Killing vector fields in this case.

\begin{lemma} If $(G,g)$ has constant non-zero curvature, then no non-zero left-invariant vector field is a Killing vector field.
\end{lemma}
\begin{proof}
The curvature hypothesis implies that $D = \lambda \Id+A$  with $\lambda\neq 0$ and $A$ skew-symmetric. In particular $D$ is invertible. Assuming that $x = \gamma b + h \in \mg$, with $\gamma\in \mathbb{R}$ and $h\in\mh$, determines a left-invariant Killing vector field, we obtain
	\begin{multline*}
			0 = \d(x)   = \gamma \d(b) + \d(h)  =- \tfrac{\gamma}{2}D(L_{\mh}) + b\cdot D(h) \Leftrightarrow \left\{ \begin{array}{l}
				D(h) = 0 \\
				\tfrac{\gamma}{2}D(L_{\mh}) = 0
			\end{array} \right.   \Leftrightarrow \left\{ \begin{array}{l}
				h=0 \\
				\gamma = 0
			\end{array} \right. .
		\end{multline*} It then follows that $x = 0$.    
\end{proof}

Consequently, in the case of constant non-zero curvature, Question \ref{q:question} is reduced to whether any left-invariant Killing tensor can be written as a polynomial in Killing vector fields that are either right-invariant or induced by skew-symmetric derivations. We will show that the answer is negative in this case:

	\begin{theorem} \label{teo:seccte}
 If $(G,g)$ has constant non-zero sectional curvature, then the metric cannot be expressed as a polynomial in right-invariant vector fields and Killing vector fields induced by skew-symmetric derivations. 
	\end{theorem}
 	\begin{proof}  First notice that for any $T=b\wedge v + T^\mh\in \Dera(\mg)$, we must have $v=0$ and $[T^\mh,D] = 0$, due to \eqref{Tderivati}.
	
Assume that the metric can be expressed as a degree 2 polynomial in the algebraic Killing vector fields of $G$. Then there exist real constants $A_{i,j}$, $B_{i,j}$ and $C_{i,j}$ and skew-symmetric derivations $\{ T_i = T^{\mh}_i  \}_{1\leq i \leq m}\subset  \Dera(\mg) $ such that
        \begin{multline}\label{tequation}
			L= A_{0,0} \xi_b \cdot \xi_b + \sum_{1\leq i \leq n}A_{0,i} \xi_b \cdot \xi_{h_i}   + \sum_{1\leq i \leq j \leq n} A_{i,j} \xi_{h_{i}} \cdot \xi_{h_{j}} + \sum_{1\leq i\leq m} B_{0,i}\xi_b \cdot \xi_{T_{i}} \\ + \sum_{\substack{1\leq i \leq n\\ 1\leq j\leq m}} B_{i,j} \xi_{h_{i}}\cdot \xi_{T_{j}} + \sum_{1\leq i\leq j \leq m} C_{i,j} \xi_{T_{i}}\cdot\xi_{T_{j}}.
		\end{multline} 
 Recall that for every $x\in\mg$ and $T\in \Dera(\mg)$, $\xi_x$ denotes the left-invariant vector field on $G$ induced by $x$ and $\xi_T$ is the Killing vector field induced by $T$ (see Section \ref{sec:leftinv}). We are going to compute $\Omega_L(\gamma b+h)$, for $\gamma\in \mathbb{R}$ and $h\in\mh$, using the right hand side of this equality, which has to coincide with $L$, since the metric is left-invariant. By linearity of $\Omega$ and \eqref{omegvf}, we compute $\Omega_L$ using the formulas \eqref{Omegahi}, \eqref{Omegab} and \eqref{OmegaTt} (for $v=0$), which gives
  \begin{eqnarray*}
      \Omega_{T_i}(\gamma b + h) =  T^{\mh}_i(h) - \tfrac{\gamma}{2}D T^{\mh}_i(h) + \tfrac{\gamma^2}{3!} D^2T^{\mh}_i(h)-\ldots.
  \end{eqnarray*} For $\gamma = 0$ and $h=0$, one has
		\[
		 b^2 + L_{\mh} =\Omega_{L}(0) = A_{0,0} b\cdot b +   \sum_{1\leq i \leq n} A_{0,i} b\cdot h_{i} + \sum_{1\leq i \leq j \leq n} A_{i,j} h_{i} \cdot h_{j} .
		\] This implies $A_{i,j} = \delta_{i,j}$, for all $0\leq i\leq j\leq n$, where $\delta_{i,j}=1$ if $i=j$ and zero otherwise. Now, for $h=0$ and arbitrary $\gamma$ we get
		\[ 
		\Omega_{L} (\gamma b) =   b^2   + e^{-\gamma D  }(L_{\mh})   = b^2  +  L_{\mh} -\gamma D(L_{\mh}) +\tfrac{\gamma^2}{2}D^2(L_{\mh})-\ldots .
		\]
		The above expression is equal to $L$ if and only if $D(L_{\mh}) =0$. By the condition ${D(L_{\mh})  = 4S_D }$ of Proposition \ref{properties}, it would then imply that $D$ is skew-symmetric and thus $\lambda = 0$, which contradicts the curvature assumption.	Therefore, $L$ cannot be written as in \eqref{tequation}.  \end{proof}

The last theorem implies, by Takeuchi and Thompson's result, that negatively curved almost abelian Lie groups carry Killing vector fields that are not algebraic:
\begin{corollary}\label{cor:Kvf}
     The Lie algebra of Killing vector fields of a connected almost abelian Lie group with left-invariant metric of constant non-zero sectional curvature contains strictly the vector space spanned by Killing vector fields induced by skew-symmetric derivations and by right-translations. 
 \end{corollary}

At the Lie group level, this corollary has implications on the structure of the isometry group. For instance, if $G$ is simply connected, the group $H:=\Auto(G)\ltimes G$ is a Lie subgroup of $\Iso(G,g)$ where, for $f\in \Auto(G)$ and $u\in G$, the action of $f$ on $\mathcal L_u$ is $\mathcal L_{f(u)}$. The Lie algebra of $H$ is the semidirect product $\Dera(\mg)\ltimes \mg$, and is the span of the Killing vector fields that are induced by right-translations of elements in $\mg$ or by skew-symmetric derivations as described in Section \ref{sec:leftinv}. Corollary \ref{cor:Kvf} in particular implies that $H$ is strictly contained in $\Iso(G,g)$ when $G$ is simply connected almost abelian with constant non-zero curvature.

	\bibliography{biblio}
	\bibliographystyle{plain}

\end{document}